\documentclass[11pt]{amsart}
\usepackage{amsmath,amssymb,amsthm,upref,graphicx,mathrsfs}

\usepackage{color,xcolor}
\usepackage[
  colorlinks=true,
  linkcolor=blue,
  citecolor=blue,
  urlcolor=blue]{hyperref}

\numberwithin{equation}{section}

\newtheorem{thm}{Theorem}[section]
\newtheorem{cor}[thm]{Corollary}
\newtheorem{lem}[thm]{Lemma}
\newtheorem{prop}[thm]{Proposition}

\newtheorem{rem}[thm]{\bf{Remark}}

\numberwithin{equation}{section}

\let \a=\alpha

\begin{document}

\leftline{ \scriptsize}

\vspace{1.3 cm}
\title
{  Orthogonal bases of Brauer symmetry classes of tensors for
 groups having cyclic support on non-linear Brauer characters}
\author{M. Hormozi and K. Rodtes}
\thanks{{\scriptsize
\hskip -0.4 true cm MSC(2000):Primary   20C30; Secondary 15A69
\newline Keywords: Brauer symmetry classes of tensors , Orthogonal basis , Semi-dihedral groups, Dicyclic groups
}}
\hskip -0.4 true cm

\maketitle


\begin{abstract}
This paper provides some properties of Brauer symmetry classes of tensors. We derive a dimension formula for the orbital subspaces in the Brauer symmetry classes of tensors corresponding to the irreducible Brauer characters of the groups having cyclic groups support on non-linear Brauer characters. Using the derived formula, we investigate the necessary and sufficient condition for the existence of the o-basis of Dicyclic groups, Semi-dihedral groups and also reinvestigate those things on Dihedral groups. Some criteria for the non-vanishing elements in the Brauer symmetry classes of tensors associated to those groups are also included.
\end{abstract}

\vskip 0.2 true cm


\pagestyle{myheadings}
\markboth{\rightline {\scriptsize  Mahdi Hormozi and Kijti Rodtes}}
         {\leftline{\scriptsize }}
\bigskip
\bigskip


\vskip 0.4 true cm

\section{Introduction}
During the past decades, there are many papers devoted to study symmetry classes of tensors, see, for example, [1-9].  One of the active research topics is the investigation of the special basis (o-basis) for the classes.  This basis consists of decomposable symmetrized tensors which are images of the symmetrizer using an irreducible character of a given group.  In \cite{HK}, Randall R. Holmes and A. Kodithuwakku studied symmetry classes of tensors using an irreducible
Brauer character of the Dihedral group instead of an ordinary irreducible character and gave necessary and sufficient conditions
for the existence of an o-basis. A classical method to find the conditions applies the dimension of the orbital subspaces in order to find an o-basis for each orbit separately. A main tool for computing the dimension of symmetry classes using ordinary characters is the Freese's theorem \cite{F2}.  Unfortunately, the symmetrizer using Brauer characters is not (in general) idempotent, so the Freese's theorem can not be applied directly.  However, for the case of non-linear Brauer characters of Dihedral groups, the authors in \cite{HK} decomposed them into a sum of ordinary characters and used generalized Freese's theorem to bound the dimension.

 One common property for all non-linear Brauer characters of Dihedral groups is the vanishing outside some cyclic subgroups. Many finite groups, including Dicyclic groups and Semi-dihedral groups, satisfy this property. In this paper, we investigate the existence of an o-basis of Brauer symmetry classes of tensors associated with the groups having the stated property. Some general properties of Brauer symmetry classes of tensors are stated. For the non-linear case, we decompose the orbital subspaces of Brauer symmetry classes of tensors into the orthogonal direct sum of smaller factors and then provide a  dimension formula for each of them. The necessary and sufficient condition for the existence of o-basis for Dicyclic groups, Semi-dihedral groups and Dihedral groups are investigated and reinvestigated as an application of the formula.  Some criteria for the non-vanishing elements in the Brauer symmetry classes of tensors associated to these groups are also included.

\section{Preliminaries}
Let $G$ be a subgroup of the full symmetric group $S_{m}$ and $p$ be a fixed prime number. An element of $G$ is \textit{$p$-regular} if its order
is not divisible by $p$. Denote by $\hat{G}$
the set of all $p$-regular elements of $G$.
Let IBr$(G)$ denote the set of irreducible Brauer characters of $G$. A Brauer
character is a certain function from $\hat{G}$ to $\textbf{C}$ associated with an $FG$-module
where $F$ is a suitably chosen field of characteristic $p$. The Brauer character
is irreducible if the associated module is simple. A conjugacy class of $G$ consisting of $p$-regular elements is called a $p$-regular class. The number of irreducible Brauer characters of $G$ equals the number of $p$-regular classes of $G$.
Let Irr$(G)$ denote the set of irreducible characters of $G$. (Unless preceded
by the word "Brauer", the word "character" always refers to an ordinary character.) If the order of $G$ is not divisible by $p$, then $\hat{G}=G$
and IBr$(G)$ =Irr$(G)$.
Let $S$ be a subset of $G$ containing the identity element $e$ and let $\phi: S \rightarrow \mathbb{C}$
be a fixed function. Statements below involving $\phi$ will hold in particular if $\phi$ is a character of $G$ (in which case $S = G$) and also if $\phi$ is a Brauer character
of $G$ (in which case $S = \hat{G}$). During the last few years, many very interesting results on the topic of Brauer
characters have been found (see e.g. \cite{Bra1} and \cite{Bra2}- \cite{Bra9}).
\\

Let $V$ be a $k$-dimensional complex inner product space and $\{e_1, . . . , e_k \}$ be an orthonormal basis of $V$. Let $\Gamma^{m}_k$ be the set of all sequences $\alpha = (\alpha_1,...,\alpha_m)$, with $1  \leq \alpha_i\leq k$. Define the action of $G$ on $\Gamma^{m}_k$ by
$$
\alpha\sigma =(\alpha_{\sigma(1)},..., \alpha_{\sigma(m)} ).
$$
We denote by $G_\alpha$ the \textit{stabilizer subgroup} of $\alpha$, i.e. $G_\alpha= \{\sigma \in G| \alpha\sigma=\alpha\}$.   The space $V^{\otimes m}$ is a left $\mathbb{C}G$-module with the action given $\sigma e_{\gamma}=e_{\gamma\sigma^{-1}}$ ($\sigma \in G, \gamma \in \Gamma^{m}_{k}$) extended linearly.  The inner product on $V$ induces an inner product on $V^{\otimes m}$ which is $G$-invariant and, with respect to this inner product, the set $\{e_{\alpha} | \alpha \in \Gamma^{m}_k\}$ is an orthonormal basis for $V^{\otimes m}$, where $e_\alpha=e_{\alpha_1}\otimes\cdots\otimes e_{\alpha_m}$.\\

The \textit{symmetrizer} corresponding to $\phi$ and $S\subseteq G$ is the element $s_\phi$ of the group algebra $\mathbb{C}G$ given by

\begin{equation}
\label{e1}
s_\phi= \frac{\phi(e)}{|S|}\sum_{\sigma \in S}\phi(\sigma)\sigma.
\end{equation}
Corresponding to $\phi$ and $\alpha \in \Gamma^{m}_k$, the \textit{standard (or decomposable)symmetrized tensor} is
\begin{equation}
\label{e1111}
e^{\phi}_{\alpha}=s_\phi e_{\alpha}= \frac{\phi(e)}{|S|}\sum_{\sigma \in S}\phi(\sigma)e_{\alpha\sigma^{-1}}.
\end{equation}
The \textit{symmetry class of tensors} associated with $\phi$ and $S\subseteq G$ is
$$
V_{\phi}(G)=s_\phi V^{\otimes m}= \langle e^{\phi}_{\alpha} | \alpha \in \Gamma^{m}_k\rangle.
$$
If $\phi$ is a Brauer character, we refer to this as a \textit{Brauer symmetry class of tensors}.
The orbital subspace of $V_\phi(G)$ corresponding to $\alpha \in \Gamma^{m}_k$ is
$$
V^{\phi}_\alpha(G)=\langle e^{\phi}_{\alpha \sigma} | \sigma \in G \rangle.
$$
An \emph{o-basis} of a subspace $W$ of $V_\phi(G)$ is an orthogonal basis of $W$ of the
form $\{e^{\phi}_{\alpha_1},\dots,e^{\phi}_{\alpha_ t}\}$ for some $\alpha_i \in \Gamma^{m}_k$.  By convention, the empty set is
regarded as an o-basis of the zero subspace of $V_\phi(G)$.  Let $\Delta = \Delta_G$ be a set of representatives of the orbits of $\Gamma^{m}_k$ under the
action of $G$.

The following is a main theorem used to reduce the task of investigation on the existence of an o-basis.
\begin{thm}\label{thm111} We have an orthogonal  sum decomposition
\label{sum}
$$
V_{\phi}(G)=  \sum_{\alpha \in \Delta} V^{\phi}_\alpha(G).
$$
\end{thm}
\begin{proof} See \cite[Thm. 1.1]{HK}.
\end{proof}
The induced inner product on $V_{\phi}(G)$ can be calculated the formula below, which is an adaptation from the Theorem 1.2 in \cite{HK}.
\begin{thm}
\label{t3}
 For every $\alpha \in \Gamma^{m}_k$ and $\sigma_1, \sigma_2\in G$ we have
  \begin{equation}\label{innert3}
 \langle e^{\phi}_{\alpha \sigma_1}, e^{\phi}_{\alpha \sigma_2} \rangle = \frac{|\phi(e)|^2}{|S|^2}\sum_{\mu \in S} \sum_{  \tau \in  \sigma_{1}\mu^{-1} S \sigma^{-1}_2  \cap G_{\alpha}} \phi(\mu) \overline{\phi(\mu\sigma^{-1}_{1} \tau\sigma_{2})}.
\end{equation}
\end{thm}
\begin{proof} For $\alpha \in \Gamma^{m}_k$ and $\sigma_1,\sigma_2 \in G$, we have

\begin{eqnarray*}
 \langle e^{\phi}_{\alpha \sigma_1}, e^{\phi}_{\alpha \sigma_2} \rangle &=& \frac{|\phi(e)|^2}{|S|^2}\sum_{\mu \in S} \sum_{\rho \in S} \phi(\mu) \overline{\phi(\rho)}  \langle e_{\alpha\sigma_1  \mu^{-1}}, e_{\alpha\sigma_2  \rho^{-1}}\rangle \\
    &=& \frac{|\phi(e)|^2}{|S|^2}\sum_{\mu \in S} \sum_{\rho \in S} \phi(\mu) \overline{\phi(\rho)}  \langle e_{\alpha \sigma_1 \mu^{-1} \rho \sigma_{2}^{-1}  }, e_{ \alpha } \rangle\\
     &=& \frac{|\phi(e)|^2}{|S|^2}\sum_{\mu \in S} \sum_{\rho \in S, \sigma_1 \mu^{-1}\rho \sigma_{2}^{-1}    \in G_{\alpha}} \phi(\mu) \overline{\phi(\rho)}\\
     &=& \frac{|\phi(e)|^2}{|S|^2}\sum_{\mu \in S} \sum_{  \tau \in \sigma_{1}\mu^{-1} S \sigma^{-1}_2  \cap G_{\alpha}} \phi(\mu) \overline{\phi(\mu\sigma^{-1}_{1}\tau \sigma_{2})},
\end{eqnarray*}
where $\tau= \sigma_1 \mu^{-1} \rho \sigma_{2}^{-1} $.
\end{proof}

The following is immediate.
\begin{cor}
\label{c2}
 Let $\sigma_{1},\sigma_{2} \in G, S\subseteq G$ and $\phi=\psi\mid_{S}$, where $\psi$ is a linear character of $G$. If $G_{\alpha}=\{ e\}$ and  $A=\{ \mu \in S \mid e \in \sigma_{1}\mu^{-1}S\sigma_{2}^{-1}\}\neq \emptyset$, then
$$
\langle e^{\phi}_{\alpha \sigma_1}, e^{\phi}_{\alpha \sigma_2} \rangle \neq 0.
$$
\end{cor}

In the following sections, we also need the lemma and propositions below.
\begin{lem}\label{l1}
 Assume that $S$ is closed under conjugation by elements of $G$
and that $\phi$ is constant on the conjugacy classes of $G$. For each $\alpha \in \Gamma^{m}_k$ and $\sigma \in G$, we have $\sigma e^{\phi}_\alpha= e^{\phi}_{\alpha \sigma^{-1}}$.
\end{lem}
\begin{proof} See \cite[Lem. 1.3]{HK}.
\end{proof}
As a consequence of this lemma, we have the following proposition.
\begin{prop}\label{rebasis}  Let $\phi:S\longrightarrow \mathbb{C}$ be a fixed function equipped the assumption of Lemma \ref{l1}.  If $B=\{e^{\phi}_{\alpha g_{1}} , e^{\phi}_{\alpha g_2} ,...,e^{\phi}_{\alpha g_k} \}$ is an o-basis of $V^{\phi}_{\alpha}(G)$, then,  for each $g\in G$, $$g\cdot B=\{e^{\phi}_{\alpha g_{1}g^{-1}}, e^{\phi}_{\alpha  g_{2}g^{-1}},...,e^{\phi}_{\alpha  g_{k}g^{-1}} \}$$ is also an o-basis of $V^{\phi}_{\alpha}(G)$.
\end{prop}
\begin{proof} This is an immediate result of Lemma \ref{l1}.
\end{proof}

\begin{prop}\label{separate} Let $\phi:S\longrightarrow \mathbb{C}$ be a fixed function.  Also, let $C$ contained in $S$ be a subgroup of $G$.   If $G_{\gamma}=\{ e\}$ and $\phi(s)=0$ for all $s\in S-C$, then $$ V^{\phi}_{\gamma}(G)=\langle e^{\phi}_{\gamma g} \mid g \in C \rangle\oplus \langle e^{\phi}_{\gamma g} \mid g \in G- C \rangle.  $$
\end{prop}
\begin{proof}  If we choose $\sigma_1 \in C$ and $\sigma_2 \in G\setminus C$ in (\ref{innert3}), we get non-zero term only if $\mu \in C$ and $\mu\sigma_1^{-1}\sigma_2\in C$, which is impossible, since $C$ is a group.  Thus, the two spaces are orthogonal.
\end{proof}

\begin{prop}\label{prop111} Let $S$ be a subgroup of $G$ and $\phi:S\longrightarrow \mathbb{C}$ be a non zero constant function on $S$.  Then, for each $\alpha \in \Gamma^{m}_{\dim V}$, $$V^{\phi}_{\alpha}(G)=\langle e^{\phi}_{\alpha \sigma} | \sigma \in G \rangle$$ has an o-basis and so does $V_{\phi}(G)$.
\end{prop}
\begin{proof}  Suppose $\phi(s)=c \in \mathbb{C}$ for all $s\in S$.   Since $S$ is a group and by Theorem \ref{t3}, we have that, for $\sigma, \tau \in G$,

 $$
  \langle e^{\phi}_{\alpha \sigma}, e^{\phi}_{\alpha \tau} \rangle =\frac{|c|^2}{|S|^2}\sum_{\mu \in S} \sum_{  \delta \in \sigma\mu^{-1} S \tau^{-1}\cap G_{\alpha}} |c|^2=\frac{|c|^4}{|S|^2}\sum_{\mu \in S} |\sigma S \tau^{-1} \cap G_{\alpha}|=\frac{|c|^4|\sigma S \tau^{-1}\cap G_{\alpha}|}{|S|}.
  $$
   We have $G_{\alpha}\cap \sigma S \tau^{-1}= \emptyset$ or $G_{\alpha}\cap \sigma S \tau^{-1}\neq \emptyset$, for each $\sigma, \tau \in G$.  For the latter case, we have $\sigma\mu \tau^{-1}\in G_{\alpha}$ for some $\mu \in S$.  Thus, for each $b \in S$,
  $$ \alpha\sigma b=\alpha(\sigma\mu \tau^{-1})(\tau \mu^{-1}b)=\alpha \tau g, \hbox{ for some $g=\mu^{-1}b \in S$.} $$
  Hence, $\{ \alpha \sigma b | b \in S\} = \{ \alpha\tau b| b \in S\}$ which means that $$e_{\alpha\sigma}^{\phi}=\frac{c^2}{|S|}\sum_{s\in S}e_{\alpha s}=e_{\alpha \tau}^{\phi}$$ since $S$ is a group and $\phi(s)=c$ for all $s \in S$.
  This implies that, for  $\sigma, \tau \in G$, $e_{\alpha\sigma}^{\phi}=e_{\alpha \tau}^{\phi}$ or $\langle e^{\phi}_{\alpha \sigma}, e^{\phi}_{\alpha \tau } \rangle =0$, which yields that $V^{\phi}_{\alpha}(G)$ has an o-basis and by Theorem \ref{thm111}, we complete the proof.

\end{proof}
\section{Dimension formula}
In this section, we let $G$ be a finite group having $C\subseteq S$ as a group support on $\phi$ ; i.e., $\phi(S\setminus C)=0$ and $\phi(\sigma)\neq 0$ for each $\sigma \in C \leq G$.  Thus, the induced inner product (\ref{innert3}) becomes
\begin{equation}\label{inner C}
     \langle e^{\phi}_{\alpha \sigma_1}, e^{\phi}_{\alpha \sigma_2} \rangle = \frac{|\phi(e)|^2}{|S|^2}\sum_{\mu \in C} \sum_{  \tau \in  \sigma_{1}C \sigma^{-1}_2  \cap G_{\alpha}} \phi(\mu) \overline{\phi(\mu\sigma^{-1}_{1} \tau\sigma_{2})},
\end{equation}
for every $\alpha \in \Gamma^{m}_k$ and $\sigma_1, \sigma_2\in G$.  If $ \sigma_{1}C \sigma^{-1}_2  \cap G_{\alpha}=\emptyset$, then $  \langle e^{\phi}_{\alpha \sigma_1}, e^{\phi}_{\alpha \sigma_2} \rangle =0$.  This motivates us to define a relation on $G$: for each $\alpha \in \triangle$,
\begin{equation}\label{relation1}
    \sigma_1 \sim_\alpha \sigma_2 \Longleftrightarrow \sigma_1 \in G_\alpha\sigma_2C,
\end{equation}
for all $\sigma_1, \sigma_2 \in G$.  It is not hard to check that $\sim_\alpha$,  for each $\alpha \in \triangle$, is an equivalent relation.

 Now, we set $[\sigma]$ as the equivalent class containing $\sigma$, $R^G_\alpha$ the set of representative of $G/\sim_\alpha$ and $V_\alpha^\phi([\sigma]):=\langle e^{\phi}_{\alpha g} | g \in [\sigma] \rangle$.  It is clear that $V_\alpha^\phi([\sigma])$ is a subspace of $V_\alpha^\phi(G)$.
\begin{lem}\label{separate p1} The space $V_\alpha^\phi(G)$ has an o-basis if and only if  for each $\sigma\in R^G_\alpha$, $V_\alpha^\phi([\sigma])$ has an o-basis.
\end{lem}
\begin{proof} Suppose that  $V_\alpha^\phi([\sigma])$ has an o-basis, for each $\sigma\in R^G_\alpha$.  To show that  the space $V_\alpha^\phi(G)$ has an o-basis, it suffices to prove that $V_\alpha^\phi([\sigma])$'s are orthogonal.  Now, let $\sigma_1, \sigma_2 \in G$ and $\alpha \in \Delta$.  If $ \sigma_{1}C \sigma^{-1}_2  \cap G_{\alpha}\neq\emptyset$, then $\sigma_1 \in G_\alpha\sigma_2C$.  Hence, if $\sigma_1\nsim_\alpha \sigma_2$, then $ \sigma_{1}C \sigma^{-1}_2  \cap G_{\alpha}=\emptyset$.  In other words, if $[\sigma_1]\neq [\sigma_2]$, then $\langle e^\phi_{\alpha\sigma_1},e^\phi_{\alpha\sigma_2}\rangle=0$.  The other implication is clear.
\end{proof}

For the following propositions,  denote $\langle e^{\phi}_{{\color{red}\gamma} g} | g \in C \rangle$ by $V^\phi_\gamma(C)$.
\begin{prop}\label{separated lemma} The space $V_\phi(G)$ has an o-basis if and only if for each $\gamma \in \Delta$, $V^\phi_\gamma(C)$ has an o-basis.
\end{prop}
\begin{proof} For each $[\sigma ]\in G/\sim_\alpha$, we have that
$$\begin{array}{ccl}
  V_\alpha^\phi([\sigma]) & = & \langle e^{\phi}_{\alpha g} | g \in [\sigma] \rangle \\
  & = & \langle e^{\phi}_{\alpha g} | g \in G_\alpha\sigma C \rangle \\
   & = & \langle e^{\phi}_{\alpha \sigma h} | h \in C \rangle \\
   & = & \langle e^{\phi}_{\gamma h} | h \in C \rangle; \gamma=\alpha\sigma \\
   & = & V^\phi_\gamma(C).
\end{array}$$
By Lemma \ref{separate p1} and (\ref{sum}), we finish the proof.
\end{proof}
To determine the dimension of $V^\phi_\gamma(C)$, for each $\gamma \in \Delta$, we introduce a relation $\sim^\ast_\gamma$ on $C$ by; for each $\sigma_1,\sigma_2 \in C$,
\begin{equation}\label{rela2}
    \sigma_1\sim^\ast_\gamma \sigma_2 \Longleftrightarrow \sigma_1\sigma_2^{-1}\in G_\gamma.
\end{equation}
It is obvious that $\sim^\ast_\gamma$ is an equivalent relation.  Now, we have;
\begin{prop}\label{prop dim} If $C/\sim^\ast_\gamma=\{ [\sigma_1], [\sigma_2],...,[\sigma_{t_\gamma}]\}$, then $\dim(V^\phi_\gamma(C))=\operatorname{rank} (M_{\gamma}),$
where $(M_\gamma)_{ij}:=\sum_{h\in C\cap G_\gamma}\phi(h\sigma_i\sigma_j)$ and $1\leq i,j\leq t_\gamma$.
\end{prop}
\begin{proof} For each $j\in \{1,2,...,t_\gamma\}$ and $g_1,g_2 \in [\sigma_j]$, we have that $g_1=cg_2$, for some $c \in G_\gamma$. Thus
$$ \begin{array}{ccl}
     e^\phi_{\gamma g_1} & = & \frac{\phi(e)}{|S|}\sum_{\sigma \in C}\phi(\sigma g_1)e_{\gamma\sigma^{-1} }\\
      & = & \frac{\phi(e)}{|S|}\sum_{\sigma \in C}\phi(\sigma cg_2)e_{\gamma\sigma^{-1}} \\
      & = & \frac{\phi(e)}{|S|}\sum_{\tau c^{-1} \in C}\phi(\tau g_2)e_{\gamma c\tau^{-1}}; \tau:=\sigma c \\
      & = & \frac{\phi(e)}{|S|}\sum_{\tau  \in C}\phi(\tau g_2)e_{\gamma \tau^{-1}} = e^\phi_{\gamma g_2}.
   \end{array}
 $$
 Hence, $V_\gamma^\phi(C)  =  \langle e^{\phi}_{\gamma \sigma_j} | j=1,2,...,t_\gamma \rangle$.
 Moreover, note that $e_{\gamma g_1^{-1}}=e_{\gamma g_2^{-1}}$ if $g_1, g_2 \in [\sigma_i]$. This yields
 $$ e^\phi_{\gamma g} = \frac{\phi(e)}{|S|}\sum^{t_\gamma}_{i=1}\left(\sum_{\sigma\in[\sigma_i]}\phi(\sigma g)\right)e_{\gamma \sigma_i^{-1}},$$
for each $g\in C$.  However, $\sum_{\sigma\in[\sigma_i]}\phi(\sigma g)=\sum_{h\in C\cap G_\gamma}\phi(h\sigma_ig)$. So, we have
$$ e^\phi_{\gamma \sigma_j} = \frac{\phi(e)}{|S|}\sum^{t_\gamma}_{i=1}\left(\sum_{h\in C\cap G_\gamma}\phi(h\sigma_i\sigma_j)\right)e_{\gamma \sigma_i^{-1}},~~~~  1\leq j\leq t_\gamma.$$
The result follows by $(M_\gamma)_{ij}:=\sum_{h\in C\cap G_\gamma}\phi(h\sigma_i\sigma_j)$, for $1\leq i,j\leq t_\gamma$.
\end{proof}
In particular, as a special case of Proposition \ref{prop dim}, i.e., if $C$ is a cyclic subgroup of $G$, we obtain a dimension formula for $V_\gamma^\phi(C)$.
\begin{thm}\label{dimprop} Let $C= \left< \tau \right >\subseteq S$ be a cyclic subgroup of $G$ such that $C/\sim^\ast_\gamma=\{ [\tau], [\tau^2],...,[\tau^{t_\gamma}]\}$.  Denote $v_j=\sum_{h\in C\cap G_\gamma}\phi(h\tau^{t_\gamma-j})$ and $$d_\gamma=\left|\left\{ s\in \{ 0,1,2,...,t_\gamma-1\} \mid  \sum^{t_\gamma-1}_{j=0}v_je^{\frac{2\pi sji}{t_\gamma}}=0 \right\} \right|.$$  Then $t_\gamma=\frac{|C|}{|C\cap G_\gamma |}$ and
$\dim (V^{\phi}_\gamma(C))=t_\gamma - d_\gamma $.
\end{thm}
\begin{proof}  Note that under the equivalent relation $\sim^\ast_\gamma$ with $C/\sim^\ast_\gamma=\{ [\tau], [\tau^2],...,[\tau^{t_\gamma}]\}$, we have that
$[\tau^k]=\{\sigma \in C \mid  \sigma \tau^{-k} \in G_\gamma \}=\{h\tau^k \mid h \in C\cap G_\gamma  \}$.  So, $|[\tau^k] |=|C\cap G_\gamma|$ for all $k=1,2,...,t_\gamma$ and hence
$$ t_\gamma=|C/\sim^\ast_\gamma| =\frac{|C|}{|C\cap G_\gamma|}. $$
By rank nullity theorem and Proposition \ref{prop dim}, $$\dim (V^{\phi}_\gamma(C))=\operatorname{rank}(M_{\gamma})=t_\gamma - \operatorname{nullity(M_\gamma)}=t_\gamma-d_\gamma,$$
where $d_\gamma:=\operatorname{nullity(M_\gamma)}$.

To determine $d_\gamma$, we observe that, if $C$ is a cyclic subgroup of $G$, then $M_\gamma$ can be reduced to a circulant matrix  $M^{cir}_\gamma$ by doing a bit column operation.  Precisely, $M^{cir}_\gamma=(v_0,v_1,...,v_{t_\gamma-1})$, where, for each $j=0,1,...,t_\gamma-1$, $$v_j=\sum_{h\in C\cap G_\gamma}\phi(h\tau^{t_\gamma-j}).$$   It is well known that (see e.g. \cite{Kra and Simanca}).  $$\operatorname{nullity}(M^{cir}_\gamma)=\operatorname{deg}[\operatorname{gcd}(P_v(x),x^{t_\gamma}-1)],$$
where $P_v(x)=\sum^{t_\gamma-1}_{j=0}v_jx^j$.   Note that the set of all roots (over field $\mathbb{C}$) of $x^{t_\gamma}-1$ is $U:=\{e^\frac{2\pi s i}{t_\gamma} \mid 0\leq s < t_\gamma\}$.  Thus, common factors of $P_v(x)$ and  $x^{t_\gamma}-1$ must have roots in $U$ and hence,
$$ \begin{array}{ccl}
     \operatorname{deg}[\operatorname{gcd}(P_v(x),x^{t_\gamma}-1)] & = & |\{ s\in \mathbb{Z} \mid 0\leq s < t_\gamma \hbox{ and } P_v(e^{\frac{2\pi s i}{t_\gamma}})=0 \} | \\
      & = & |\{ s\in \mathbb{Z} \mid 0\leq s < t_\gamma \hbox{ and } \sum^{t_\gamma-1}_{j=0}v_je^{\frac{2\pi sji}{t_\gamma}}=0 \} |.
   \end{array}
 $$
Since the operator $\operatorname{rank}$ is invariant under column operations, $d_\gamma=\operatorname{nullity}(M^{cir}_\gamma)$ and thus the result follows.
\end{proof}
We utilize this theorem in the following sections.

\section{Dicyclic Group  $T_{4n}$}

The Dicyclic group $T_{4n}$ is defined as follows:
$$
T_{4n}=\langle r,s | r^{2n}=e,r^n=s^2,s^{-1}rs=r^{-1} \rangle.
$$
Explicitly, all elements of the group $T_{4n}$  may be given by $T_{4n}=\{r^{i}, sr^{i} | 0\leq i < 2n \}$.
By the classical Cayley theorem, $T_{4n}$ can be embedded in $S_{4n}$.  Precisely,
\begin{eqnarray*}
  r &=& (\begin{array}{ccccc}
           1 & 2 & 3 & \cdots & 2n
         \end{array}
  )(\begin{array}{ccccc}
      2n+1 & 2n+2 & 2n+3 & \cdots & 4n
    \end{array}
  ) \\
  s &=& (\begin{array}{cccc}
           1 & 2n+1 & n+1 & 3n+1
         \end{array}
  )(\begin{array}{cccc}
           2 & 4n & n+2 & 3n
         \end{array}
  ) \\
   && (\begin{array}{cccc}
           3 & 4n-1 & n+3 & 3n-1
         \end{array}
  )\cdots (\begin{array}{cccc}
           n-1 & 3n+3 & 2n-1 & 2n+3
         \end{array}
  ) \\
  && (\begin{array}{cccc}
           n & 3n+2 & 2n & 2n+2
         \end{array}
  ).
\end{eqnarray*}
$T_{4n}$ has $n+3$ conjugacy classes which are
$$\{e \}, \{r^{k},r^{2n-k} \}, 1\leq k\leq n,\{sr^{2k} \mid 0\leq k\leq n-1 \},\{sr^{2k+1} \mid 0\leq k\leq n-1 \}  $$
and the ordinary irreducible character of $T_{4n}$ are given by (see \cite{DP})
\begin{center}
\begin{tabular}{|c|c|c|c|c|}
  \hline
  \hline

   Characters & $r^k( 0\leq k\leq n)$ & $s$ & $rs$ \\
   \hline
   \hline
  $\chi_{0}$  & 1 & 1 & 1 \\
  \hline
 $\chi_{1}$  & $(-1)^k$& 1 & -1 \\
  \hline
 $\chi_{2}$ & 1 & -1 & -1 \\
  \hline
 $\chi_{3}$ & $(-1)^k$ & -1 & 1 \\
  \hline
  &&&\\
  $\psi_{j}$, where  & $2\cos(\frac{kj\pi}{n})$ & 0 & 0 \\
  $1\leq j\leq n-1$   &  &  &  \\

  \hline
  \hline
\end{tabular}\\
\end{center}
\begin{center}
    \textbf{Table I} The character table for $T_{4n}$, where $n$ is even.
\end{center}

\begin{center}
\begin{tabular}{|c|c|c|c|c|}
  \hline
  \hline

   Characters  & $r^k( 0\leq k\leq n)$ & $s$ & $rs$ \\
   \hline
   \hline
  $\chi'_{0}$ & 1 & 1 & 1 \\
  \hline
 $\chi'_{1}$  & $(-1)^k$& i & -i \\
  \hline
 $\chi'_{2}$ & 1 & -1 & -1 \\
  \hline
 $\chi'_{3}$  & $(-1)^k$ & -i & i \\
  \hline
  &&&\\
  $\psi'_{j}$, where & $2\cos(\frac{kj\pi}{n})$ & 0 & 0 \\
  $1\leq j\leq n-1$ &    &  &  \\

  \hline
  \hline
\end{tabular}\\
\end{center}

\begin{center}
    \textbf{Table II} The character table for $ T_{4n}$, where $n$ is odd.
\end{center}

Write $2n = lp^t$ with $l$ an integer not divisible by $p$ (where $p$ is our fixed prime number). We have
$$  \hat{G}=\left\{
              \begin{array}{ll}
                \{ r^{jp^t}, sr^{k}|  0 \leq j<l ,1\leq k\leq 2n \}, & \hbox{if $p\neq2$;} \\
                \{ r^{jp^t}|  0 \leq j<l  \}, & \hbox{if $p=2$.}
              \end{array}
            \right.
 $$
Thus the $p$-regular classes of $G$ are
$$\left\{
  \begin{array}{ll}
    \{ r^{jp^t},r^{(l-j)p^t}\};\hbox{ $0\leq j\leq\frac{l}{2}$}, \{ sr^{2k}| 1\leq k\leq n \},  \{ sr^{2k+1}| 0\leq k\leq n-1 \}, & \hbox{if $p\neq2$;} \\
     \{ r^{jp^t},r^{(l-j)p^t}\}; \hbox{ $0\leq j\leq\frac{l-1}{2}$} & \hbox{if $p=2$.}
  \end{array}
\right.
$$
For each $j$ and $h$ denote
$$
\hat{\psi}_j=\psi_j|_{\hat{G}}, ~~\hat{\chi}_h=\chi_h|_{\hat{G}} \hbox{  and  }\hat{\psi'}_j=\psi'_j|_{\hat{G}}, ~~\hat{\chi'}_h=\chi'_h|_{\hat{G}},
$$
and define $\epsilon=\left\{
                   \begin{array}{ll}
                     4, & \hbox{if $p\neq 2$;} \\
                     1, & \hbox{if $p=2$.}
                   \end{array}
                 \right.
$
\begin{prop}  The complete list of irreducible Brauer characters of $T_{4n}$ for even $n$ is
 $$\hat{\chi}_h~~(0\leq h < \epsilon),~~ \hat{\psi}_{j} ~~(1\leq j < \frac{l}{2}),$$
and for odd $n$ is
$$\hat{\chi'}_h ~~(0\leq h < \epsilon), ~~\hat{\psi'}_{j} ~~(1\leq j < \frac{l}{2}).$$
\end{prop}
\begin{proof} We first note that the restriction of a character of $T_{4n}$ to $\widehat{T}_{4n}$ is a Brauer character and the number of all the irreducible Brauer characters is the number of $p$-regular classes of $T_{4n}$.   Also, since $T_{4n}$ is solvable, by the Fong-Swan theorem, any irreducible Brauer character of $T_{4n}$ is the restriction of an ordinary irreducible character of $T_{4n}$.

The linear characters $\widehat{\chi}_{h}$'s and $\widehat{\chi'}_{h}$'s are obviously irreducible and distinct, by the character table above.  For the characters of dimension two, $\widehat{\psi}_{j}$ and $\widehat{\psi'}_{j}$, we claim that they are all distinct and irreducible for all $1\leq j < \frac{l}{2}$.  By the character tables above, there is no need to separate the proof into the case of odd $n$, even $n$ or $p=2$, $p\neq2$, since $\widehat{\psi}_{j}$ and $\widehat{\psi'}_{j}$ are agree on the columns $r^k$'s and agree to be zero outside these columns.

For the irreduciblity issue, we suppose for a contradiction that $$\widehat{\psi}_{j}=\widehat{\chi}_{h}+\widehat{\chi}_{k},$$ for some $0\leq h,k < \epsilon$ and $1\leq j < \frac{l}{2}$.
Since $1\leq j < \frac{l}{2}$, $l>2$ and $r^{2p^t}\in \widehat{T}_{4n}$.  So, we can evaluate both sides of the above equation by $r^{2p^t}$ and obtain that
$$2\cos(\frac{2p^tj\pi}{n})=2,$$
 which is impossible because $\cos(\frac{2p^tj\pi}{n})<1$ for all  $1\leq j < \frac{l}{2}$.

Analogously, for the issue of distinction, we suppose for a contradiction that $\widehat{\psi}_{j}=\widehat{\psi}_{i}$, for some $1\leq i < j <\frac{l}{2}$.  We now evaluate both sides by $r^{p^t}$, which yields that $$\cos(\frac{p^tj\pi}{n})=\cos(\frac{p^ti\pi}{n}).$$  It implies that, for $\frac{p^tj\pi}{n}$ and $\frac{p^tj\pi}{n}$, their difference or their sum must be a multiple of $2\pi$.  However, this is not the case because $1\leq i < j <\frac{l}{2}$.

\end{proof}

\begin{thm}
 \label{t26}
 Let $G=T_{4n}$, $0 \leq h<\epsilon$ where $\epsilon=4$ if $p\neq2$ and $\epsilon=1$ if $p=2$,  and put $\phi=\hat{\chi}_h$ or $\hat{\chi'}_h $ . The space $V_\phi(G)$ has an o-basis if and only if at least one of the following holds:\\
(i) $\dim V=1$\\
(ii) $p = 2$,\\
(iii) $2n$ is not divisible by $p$.
\end{thm}
\begin{proof}
(i) If $\dim V=1$, then $V_\phi(G)=\langle e^{\phi}_{\alpha} \mid \alpha\in \Gamma^{4n}_{1}\rangle$ has only at most one generator, namely, $e^{\phi}_{\alpha}$ where $\alpha=(1,1,1...,1)$.  So, $\dim V_\phi(G) \leq 1$ and thus $V_\phi(G)$ has an o-basis.\\
(ii) If $p=2$ then $\widehat{G}=\langle r^{p^t}\rangle$. Since $\widehat{G}$ is a subgroup of $G$ and $\phi$ is constant on $\widehat{G}$ , it  follows by Proposition \ref{prop111} that $V_\phi(G)$ has an o-basis.\\
(iii)  Assume $p\neq 2$ and $2n$ is not divisible by $p$. Then $\hat{G}=G$ and consequently, these characters will be ordinary linear characters. Thus $V_\phi(G)$ has an o-basis.

Conversely, we assume that  $\dim V >1$ and $p\neq 2$ and $2n$ is divisible by $p$.  So, $r \notin \widehat{G}$ and $\widehat{G}=\{r^{jp^{t}}, sr^{k} \mid 0\leq j<l, 1\leq k \leq 2n \}=\widehat{G}^{-1}$.   We will show that $V_{\phi}(G)$ does not have an o-basis. For $\alpha=(1,2,..,2,2)  \in \Gamma^{4n}_{\dim V}$, we have $G_{\alpha}=\{e\}$. Now, we concentrate on $\langle e^{\phi}_{\alpha\sigma}, e^{\phi}_{\alpha}\rangle$, for each $\sigma \in G$.  We observe that  $A=\{ \mu \in \widehat{G} \mid e \in \sigma\mu^{-1}\widehat{G}\}=\{ \mu \in \widehat{G} \mid \sigma \in\widehat{G}\mu\}$.  Since $r^{i}=(sr^n)(sr^{i})\in \widehat{G}^{2}$ for each $0\leq i<2n$, $G\subseteq\widehat{G}^{2}$ and hence $A\neq\emptyset$.  Thus by Corollary \ref{c2}, we have
\begin{equation}\label{eqp111}
    \langle e^{\phi}_{\alpha \sigma }, e^{\phi}_{\alpha} \rangle \neq 0 \hbox{ for each } \sigma \in G.
\end{equation}

 Next, we claim that $\{ e^{\phi}_{\alpha r},  e^{\phi}_{\alpha}\}\subseteq V^{\phi}_{\alpha}(G)$ is a linearly independent set.  We can set $e^{\phi}_{\alpha}= \sum_{\delta}c_{\delta}e_{\delta}$ and $e^{\phi}_{\alpha r}= \sum_{\delta}d_{\delta}e_{\delta}$ as $\{e_{\delta}| \delta \in \Gamma^{4n}_{\dim V} \}$ forms a basis for $V^{\otimes 4n}$.  Since $\widehat{G}^{-1}=\widehat{G}$, $$e^{\phi}_{\alpha}=\frac{\phi(1)}{|\widehat{G}|}\sum_{\sigma\in \widehat{G}}\phi(\sigma^{-1})e_{\alpha\sigma}.$$
Since $G_{\alpha}=\{e\}$, the elements $\alpha \sigma$ with $\sigma \in G$ are distinct.  Also, since $r \notin \widehat{G}$, $\alpha\sigma \neq\alpha r$ for all
$\sigma\in\widehat{G}$, which yields that $c_{r}=0$. On the other hand,  $G_{\alpha r}=r^{-1} G_{\alpha} r=\{ e\}$, so for $r \in \widehat{G}$, $(\alpha r)\sigma=\alpha r$ if
and only if $\sigma=e$.  This implies that $d_{r}=\frac{1}{|\widehat{G}|}\neq c_{r}=0$, which implies that $\{ e^{\phi}_{\alpha r},  e^{\phi}_{\alpha}\}\subseteq V^{\phi}_{\alpha}(G)$ is a linearly independent set and hence $\dim V_{\alpha}^{\phi}(G)\geq2$.

 By Proposition \ref{rebasis}, if $V_{\alpha}^{\phi}(G)$ were to have an o-basis, then it would have an o-basis containing $e^{\phi}_{\alpha}$, but, by (\ref{eqp111}), this is not the case.  So, $V_{\alpha}^{\phi}(G)$ does not have an o-basis, and by Theorem \ref{thm111}, we complete the proof.
\end{proof}

For higher dimensional irreducible Brauer characters $\phi: \widehat{G}\longrightarrow\mathbb{C}$, we see that  if $\dim V=1$, then $V_\phi(G)=\langle e^{\phi}_{\alpha} \mid \alpha\in \Gamma^{4n}_{1}\rangle$ has only at most one generator, namely, $e^{\phi}_{\alpha}$ where $\alpha=(1,1,1...,1)$.  So, $\dim V_\phi(G) \leq 1$ and thus $V_\phi(G)$ has an o-basis.  If $\dim V>1$, we investigate a necessary condition of the existence of an o-basis for Dicyclic groups as follows.
\begin{prop}\label{dimen T} For $G=T_{2(lp^t)}$ with $C\cap G_\gamma=<r^{t_\gamma p^t}>$ where $t_\gamma=\frac{l}{|C\cap G_\gamma|}$, $\gamma\in \Delta$ and $\phi=\hat{\psi}_b, \hat{\psi'}_b $, where $ 1 \leq b< \frac{l}{2}$, we have that
$$\dim(V_\gamma^\phi(C))=\left\{
                           \begin{array}{ll}
                             2, & \hbox{if $\frac{bt_\gamma }{l}\in \mathbb{Z}$;} \\
                             0, & \hbox{if $\frac{bt_\gamma }{l}\notin \mathbb{Z}$.}
                           \end{array}
                         \right.
  $$
\end{prop}
\begin{proof}  Since $G$ has cyclic support with $C=<r^{p^t}>$ and $C\cap G_\gamma=<r^{t_\gamma p^t}>$, for each $\phi$, we can compute the dimension by using Proposition \ref{dimprop}.  By character tables and basic trigonometry identities, we compute that
$$\begin{array}{ccl}
    v_j=\sum_{h\in C\cap G_\gamma}\phi(h\tau^{t_\gamma-j}) & = & \sum_{m=1}^{l/t_{\gamma}}\phi(r^{(mt_\gamma-j)p^t}) \\
     & = & \sum_{m=1}^{l/t_{\gamma}}2\cos\left(m(\frac{2bt_\gamma}{l})\pi-(\frac{2bj}{l})\pi  \right)\\
     &=&\left\{
          \begin{array}{ll}
            2(\frac{l}{t_\gamma})\cos((\frac{2bj}{l})\pi), & \hbox{$\frac{bt_\gamma}{l}\in\mathbb{Z}$;} \\
            0, & \hbox{$\frac{bt_\gamma}{l}\notin\mathbb{Z}$.}
          \end{array}
        \right.
  \end{array}
  $$
So, if $\frac{bt_\gamma}{l}\notin\mathbb{Z}$, then $d_\gamma=t_\gamma$ and thus $\dim(V_\gamma^\phi(C))=t_\gamma-t_\gamma=0$.

For $d_\gamma$ in which $\frac{bt_\gamma}{l}\in\mathbb{Z}$, we have that $\sum^{t_\gamma-1}_{j=0}v_je^{\frac{2\pi sji}{t_\gamma}}=0$ if and only if $$\sum^{t_\gamma-1}_{j=0}2\cos((\frac{2bj}{l})\pi)\cos((\frac{2sj}{t_\gamma})\pi) \hbox{  and  } \sum^{t_\gamma-1}_{j=0}2\cos((\frac{2bj}{l})\pi)\sin((\frac{2sj}{t_\gamma})\pi),$$
are simultaneously zero.  Since $\frac{bt_\gamma}{l}\in\mathbb{Z}$, the second sum is always zero and the first sum is zero for all $0\leq s< t_\gamma$ except for $\frac{b}{l}\pm\frac{s}{t_\gamma} \in \mathbb{Z}$; (i.e. except for $s= t_\gamma - \frac{bt_\gamma}{l}$ or $s= \frac{bt_\gamma}{l}$), because $0<\frac{b}{l}+\frac{s}{t_\gamma} < 2$ and $-1< \frac{b}{l}-\frac{s}{t_\gamma} < 1$.  Hence $d_\gamma= t_\gamma -2$ and thus the results follow.
\end{proof}

There is no surprise with the assertion that $\dim(V_\gamma^\phi(C))=0$ for which $\frac{bt_\gamma}{l}\notin\mathbb{Z}$ because;
\begin{prop}\label{nonzeroT} For $G=T_{2(lp^t)}$ with $C=<r^{ p^t}>$ and $C\cap G_\gamma=<r^{t_\gamma p^t}>$ where $t_\gamma=\frac{l}{|C\cap G_\gamma|}$, $\gamma\in \Delta$ and $\phi=\hat{\psi}_b, \hat{\psi'}_b $, where $ 1 \leq b< \frac{l}{2}$, we have that, for each $\sigma \in C$,  $$ e^{\phi}_{\gamma\sigma}\ = 0 \hbox{  if and only if  }\frac{bt_\gamma}{l}\notin\mathbb{Z}.$$
\end{prop}
\begin{proof} Let $\sigma \in C$ and $\gamma \in \Delta$.  By (\ref{inner C}),
$$ \begin{array}{ccl}
      \langle e^{\phi}_{\gamma \sigma}, e^{\phi}_{\gamma \sigma} \rangle & = &  \frac{|\phi(e)|^2}{|\widehat{G}|^2}\sum_{\mu \in C} \sum_{  \tau \in  C  \cap G_{\gamma}} \phi(\mu) \overline{\phi(\mu\tau)} \\
     & = &4 \frac{|\phi(e)|^2}{|\widehat{G}|^2}\sum^l_{j=1} \sum^{l/t_{\gamma}}_{k=1}\cos(\frac{2jb\pi}{l}) \cos(2k(\frac{bt_\gamma }{l})\pi+\frac{2jb\pi}{l}) \\
    &=& \left\{
          \begin{array}{ll}
            (\frac{4l}{t_\gamma})( \frac{|\phi(e)|^2}{|\widehat{G}|^2})\sum^l_{j=1}\cos^2(\frac{2jb\pi}{l}), & \hbox{if $\frac{bt_\gamma}{l}\in\mathbb{Z}$;} \\
            0, & \hbox{if $\frac{bt_\gamma}{l}\notin\mathbb{Z}$,}
          \end{array}
        \right.
   \end{array}
 $$
which completes the proof.
\end{proof}
 Now, by the above propositions, we achieve the main conclusion.
\begin{thm}  Let $G=T_{4n}$, where $2n=lp^t$ with $l$ an integer not divisible by $p$ and let $\phi=\hat{\psi}_b, \hat{\psi'}_b $, where $ 1 \leq b< \frac{l}{2}$.  Then,   $V_{\phi}(G)$ has an o-basis if and only if $\nu_2(\frac{2b}{l})<0$.
\end{thm}
\begin{proof}  By Proposition \ref{separated lemma}, it is enough to focus on $V^{\phi}_{\gamma}(C)$.  Also, in the proof of Proposition \ref{prop dim}, we have $V_\gamma^\phi(C)  =  \langle e^{\phi}_{\gamma \sigma_j} | j=1,2,...,t_\gamma \rangle$, where $\sigma_j=r^{jp^t}$ and $t_{\gamma}=\frac{\mid C\mid}{\mid C\cap G_\gamma  \mid}$.  Again, by (\ref{inner C}) and the character tables, we compute that, for $1\leq i, j \leq t_\gamma$,
$$ \begin{array}{ccl}
      \langle e^{\phi}_{\gamma \sigma_i}, e^{\phi}_{\gamma \sigma_j} \rangle =0& \Longleftrightarrow & \sum_{g\in C}\phi(g\sigma_i)\overline{\phi(g\sigma_j)}=0 \\
      & \Longleftrightarrow & \sum^{l-1}_{k=0}\phi(r^{(k+i)p^{t}})\overline{\phi(r^{(k+j)p^{t}})}=0\\
&\Longleftrightarrow&\sum^{l-1}_{k=0}2\cos((k+i)\frac{2b}{l}\pi)\cos((k+j)\frac{2b}{l}\pi)=0\\
&\Longleftrightarrow&l\cos((i-j)\frac{2b}{l}\pi)=0.
   \end{array}
 $$
By Proposition \ref{dimen T}, $\dim(V^{\phi}_{\gamma}(C))=2$ for each $\gamma$ such that $\frac{bt_\gamma}{l} \in \mathbb{Z}$.  So, if $V_{\phi}(G)$ contains an o-basis, then there exist $\gamma$ and distinct $1\leq i, j \leq t_\gamma$ such that $\cos((i-j)\frac{2b}{l}\pi)=0$, which clearly implies that $\nu_2(\frac{2b}{l})<0$.  On the other hand, suppose $\nu_2(\frac{2b}{l})=-k$, for some $k \in \mathbb{N}$.  Then $\frac{b}{l}=\frac{m}{2^{k+1}}$ for some odd integer $m$.  Since the existence of o-basis depend on  $\gamma$ for which $\frac{bt_\gamma}{l} \in \mathbb{Z}$, $2^{k+1}$ is always a divisor of $t_\gamma$.  Thus, we can choose $i_0=2^{k-1}+1$ and $j_0=1$ so that $\cos((i_0-j_0)\frac{2b}{l}\pi)=0$.
By Proposition \ref{nonzeroT}, $e^\phi_{\gamma\sigma_{i_0}}$ and $e^\phi_{\gamma\sigma_{j_0}}$ are non zero and hence, by the above fact, $\{ e^\phi_{\gamma r^{(2^{k-1}+1)p^t}}, e^\phi_{\gamma r^{p^t}}\}$ forms an o-basis for $V^{\phi}_{\gamma}(C)$.
\end{proof}
\section{Dihedral Group  $D_{m}$}
We first collect some facts about the Brauer characters of the Dihedral groups $D_{2n}$ from \cite{HK}.   We follow the notions of \cite{HK} in this section.  A presentation of the Dihedral groups $D_{m}$ having order $2m$, is given by $D_m=<r,s \mid r^m=s^2=1, srs=r^{-1}>$.   The ordinary character table of $D_m$ is
\begin{center}
\begin{tabular}{ccc}
  \hline\hline
  Characters & $r^k$ & $sr^k$ \\
\hline
  $\psi_0$ & $1$ & $1$ \\
  $\psi_1$ & $1$ & $-1$ \\
  $\psi_2$ & $(-1)^k$ & $(-1)^k$ \\
  $\psi_2$ & $(-1)^k$ & $(-1)^{k+1}$ \\
\hline
&&\\

  $\chi_h$ & $2\cos \frac{2\pi k h}{m}$ & 0 \\
 ($1\leq h < \frac{m}{2}$)&&\\
  \hline\hline
\end{tabular}
\end{center}
\begin{center}
\textbf{Table II}: The character table of $D_m$.
\end{center}

We write $m=lp^t$, where $l$ is not divisible by prime number $p$, as before.  The set of all $p$-regular elements of $D_m$ are
$$ \widehat{D_m}=\left\{
                   \begin{array}{ll}
                     \{ r^{jp^t}, sr^k \mid 0\leq j <l, 0 \leq k < \frac{m}{2}\}, & \hbox{$p\neq 2$;} \\
                      \{ r^{jp^t} \mid 0\leq j <l \}, & \hbox{$p= 2$.}
                   \end{array}
                 \right.
 $$
The complete list of irreducible Brauer characters of $D_{m}$ is, \cite{HK},
 $$\hbox{$\hat{\psi}_j$ ($0\leq j < \epsilon$),  $\hat{\chi}_h$  ($1 \leq h < \frac{l}{2}$), }$$
where $\hat{\psi}_j=\psi_j\mid_{\widehat{D_m}}$,   $\hat{\chi}_j=\chi_h\mid_{\widehat{D_m}}$ and $$ \epsilon =\left\{  \begin{array}{ll}
                      4, & \hbox{$l$ even, $p\neq 2$;} \\
                      2, & \hbox{$l$ odd, $p\neq 2$;} \\
                      1, & \hbox{$p=2$.}
                    \end{array}
                  \right.  $$  The necessary and sufficient condition for the existence of an o-basis for Brauer characters of dimension one is provided in \cite{HK}.  Precisely, for $\phi=\hat{\psi}_j, \hat{\psi'}_j$, where $0 \leq j<\epsilon$,
the space $V_\phi(D_{m})$ has an o-basis if and only if $\dim V=1$ or $p = 2$ or  $m$ is not divisible by $p$.

The necessary and sufficient condition for the existence of an o-basis for Brauer character of dimension two for $D_{m}$ can be found in \cite{HK}.  But it can also be obtained by very similar method applied on $T_{4n}$ as we presented in $\S4$.   This is because $\phi$ has a cyclic support for each $\phi=\hat{\chi}_h$, where $0 \leq h < \frac{l}{2}$, and all values in the character tables of both groups  are consistent on $C=<r^k>$.  Thus,  by changing $m$ to $2n$ and $h$ to $b$, each step of the computation for dimensions of $V_{\gamma}^{\phi}(D_m)$ and the condition for the existence becomes the same.  This yields

\begin{thm}Let $G=D_{m}$, where $m=lp^t$ with $l$ an integer not divisible by $p$ and let $\phi=\hat{\chi}_h$, where $ 1 \leq h< \frac{l}{2}$.  Then,   $V_{\phi}(G)$ has an o-basis if and only if $\nu_2(\frac{2h}{l})<0$.  Also, for each $\sigma \in C$,  $ e^{\phi}_{\gamma\sigma}\ \neq 0 \hbox{  if and only if  }\frac{ht_\gamma}{l}\in\mathbb{Z}.$
\end{thm}
\section{Irreducible Brauer character of $SD_{8n}$}
The presentation for $SD_{8n}$ for $n\geq2$ is given by
$  SD_{8n}=<a,b \mid a^{4n}=b^{2}=e,bab=a^{2n-1}>.$
All $8n$ elements of $SD_{8n}$ may be given by
$$SD_{8n}=\{  e,a,a^{2},...,a^{4n-1},b,ba,ba^{2},...,ba^{4n-1}\}.$$
The embedding of $SD_{8n}$ into the symmetric group $S_{4n}$ is given by
$T(a)(t) := \overline{t + 1}$ and $T(b)(t) := \overline{(2n -1)t}$, where $\overline{m}$ is the remainder of $m$ divided by $4n$.
We write $4n=lp^{t}$ with prime $p$ and integer $l$ not divisible by $p$ and denote by $\widehat{SD_{8n}}$ the set of all $p$-regular elements of $SD_{8n}$.  It is not hard to see that
$$ \widehat{SD_{8n}}=\left\{
                       \begin{array}{ll}
                         \{a^{jp^{t}},ba^{k} \mid 0\leq j<l;0\leq k <4n \}, & \hbox{if $p\neq2$;} \\
                         \{ a^{jp^{t}} \mid 0\leq j < l\}, & \hbox{if $p=2$.}
                       \end{array}
                     \right.
 $$

\begin{prop}\label{conjuagacy} The $p$-regular classes of $SD_{8n}$, $n\geq2$ and $4n=lp^{t}$, are as follows;\\
\textbf{Case 1:} $p$ is odd prime.
\begin{itemize}
  \item If $n$ is even (i.e. $\frac{l}{8} \in \mathbb{Z}$), then there are $\frac{l}{2}+3$ $p$-regular classes. Precisely,
       \begin{itemize}
         \item 2 classes of size one being $\{ e\}$ and $\{a^{\frac{l}{2}p^{t}}\}$,
         \item $\frac{l}{4}-1$ classes of size two being $[a^{jp^{t}}]=\{ a^{jp^{t}},a^{(l-j)p^{t}}\};j \in \{ 2,4,6,...,\frac{l}{2}-2 \} $,
         \item $\frac{l}{8}$ classes of size two being $[a^{jp^{t}}]=\{ a^{jp^{t}},a^{(\frac{l}{2}-j)p^{t}}\};j \in \{ 1,3,5,...,\frac{l}{4}-1 \} $,
         \item $\frac{l}{8}$ classes of size two being $[a^{jp^{t}}]=\{ a^{jp^{t}},a^{(\frac{3l}{2}-j)p^{t}}\};j \in \{ \frac{l}{2}+1,\frac{l}{2}+3,...,\frac{l}{2}+\frac{l}{4}-1 \} $ and
         \item 2 classes of size $2n$ being $[b]=\{ba^{2i} \mid i=0,1,2,...,2n-1\}$ and $[ba]=\{ba^{2i+1} \mid i=0,1,2,...,2n-1\}$.
       \end{itemize}
 \item If $n$ is odd (i.e. $\frac{l}{4}$ is odd), then there are $\frac{l}{2}+6$ $p$-regular classes. Precisely,
       \begin{itemize}
         \item 4 classes of size one being $\{ e\}$, $\{a^{\frac{l}{4}p^{t}}\}$, $\{a^{\frac{l}{2}p^{t}}\}$ and $\{a^{\frac{3l}{4}p^{t}}\}$,
         \item $\frac{l}{4}-1$ classes of size two being $[a^{jp^{t}}]=\{ a^{jp^{t}},a^{(l-j)p^{t}}\};j \in \{ 2,4,6,...,\frac{l}{2}-2 \} $,
         \item $\frac{l-4}{8}$ classes of size two being $[a^{jp^{t}}]=\{ a^{jp^{t}},a^{(\frac{l}{2}-j)p^{t}}\};j \in \{ 1,3,5,...,\frac{l}{4}-2 \} $,
         \item $\frac{l-4}{8}$ classes of size two being $[a^{jp^{t}}]=\{ a^{jp^{t}},a^{(\frac{3l}{2}-j)p^{t}}\};j \in \{ \frac{l}{2}+1,\frac{l}{2}+3,...,\frac{l}{2}+\frac{l}{4}-2 \} $ and
        \item 4 classes of size $n$ being $[b]=\{ba^{4i} \mid i=0,1,2,...,n-1\}$, $[ba]=\{ba^{4i+1} \mid i=0,1,2,...,n-1\}$, $[ba^{2}]=\{ba^{4i+2} \mid i=0,1,2,...,n-1\}$ and $[ba^{3}]=\{ba^{4i+3} \mid i=0,1,2,...,n-1\}$.
       \end{itemize}
\end{itemize}
\textbf{Case 2:} $p=2$.
There are $\frac{l+1}{2}$ p-regular classes.  Precisely, there is 1 class of size one, $\{e\}$, and there are $\frac{l-1}{2}$ classes of size two, $\{a^{jp^{t}}, a^{(l-j)p^{t}} \}$; $1\leq j \leq \frac{l-1}{2}$.

\end{prop}
\begin{proof} This is a direct calculation.
\end{proof}

The ordinary irreducible character of $SD_{8n}$ are given by (see \cite{HORO})
\begin{center}
    \textbf{Table I} The character table for $SD_{8n}$, where $n$ is even.
\end{center}
\begin{center}
\begin{tabular}{|c|c|c|c|c|}
  \hline
  \hline
   Conjugacy classes,& $[a^{r}];$ & $[a^{r}];$ & $[b]$ & $[ba]$ \\
   Characters& $r \in C_{1}$ & $r \in C^{\dag}_{odd}$ &  &  \\
   \hline
   \hline
  $\chi_{0}$ & 1 & 1 & 1 & 1 \\
  \hline
 $\chi_{1}$ & 1 & 1 & -1 & -1 \\
  \hline
 $\chi_{2}$ & 1 & -1 & 1 & -1 \\
  \hline
 $\chi_{3}$ & 1 & -1 & -1 & 1 \\
  \hline
  &&&&\\
  $\psi_{h}$, where & $2\cos(\frac{hr\pi}{2n})$ & $2\cos(\frac{hr\pi}{2n})$ & 0 & 0 \\
  $h \in  C^{\dag}_{even}$ &  &  &  &  \\
  \hline
   &&&&\\
   $\psi_{h}$, where & $2\cos(\frac{hr\pi}{2n})$ & $2i\sin(\frac{hr\pi}{2n})$ & 0 & 0 \\
  $h \in  C^{\dag}_{odd}$ &  &  &  &  \\
  \hline
  \hline
\end{tabular}\\
\end{center}
\begin{center}
    \textbf{Table II} The character table for $SD_{8n}$, where $n$ is odd.
\end{center}
\begin{center}
\begin{tabular}{|c|c|c|c|c|c|c|}
  \hline
  \hline
   Conjugacy classes,& $[a^{r}];$ & $[a^{r}];$ & $[b]$ & $[ba]$ & $[ba^{2}]$ & $[ba^{3}]$ \\
   Characters& $r \in C_{1}$ & $r \in C_{2,3}^{odd}$ &  &  && \\
   \hline
   \hline
  $\chi'_{0}$ & 1 & 1 & 1 & 1 &1&1\\
  \hline
 $\chi'_{1}$ & 1 & 1 & -1 & -1 &-1&-1 \\
  \hline
 $\chi'_{2}$ & 1 & -1 & 1 & -1 &1&-1\\
  \hline
 $\chi'_{3}$ & 1 & -1 & -1 & 1 &-1&1\\
  \hline
  $\chi'_{4}$ & $(-1)^{\frac{r}{2}}$ & $i^{r}$ & 1 & $i$ &-1&$-i$\\
  \hline
 $\chi'_{5}$ & $(-1)^{\frac{r}{2}}$& $i^{r}$ & -1 & $-i$ &1&$i$ \\
  \hline
 $\chi'_{6}$ &$(-1)^{\frac{r}{2}}$ & $(-i)^{r}$ & 1 & $-i$ &-1&$i$\\
  \hline
 $\chi'_{7}$ &$(-1)^{\frac{r}{2}}$ & $(-i)^{r}$ & -1 & $i$ &1&$-i$\\
  \hline
  &&&&&&\\
  $\psi'_{h}$, where & $2\cos(\frac{hr\pi}{2n})$ & $2\cos(\frac{hr\pi}{2n})$ & 0 & 0 &0&0\\
  $h \in C^{\dag}_{even}$ &  &  &  &  &&\\
  \hline
   &&&&&&\\
   $\psi'_{h}$, where & $2\cos(\frac{hr\pi}{2n})$ & $2i\sin(\frac{hr\pi}{2n})$ & 0 & 0 &0&0\\
  $h\in C^{\dag}_{odd}$ &  &  &  &  &&\\
  \hline
  \hline
\end{tabular}\\
\end{center}
where $C_{1}=\{0,2,4,...,2n \}$, $C^{\dag}_{even}:=C_{1}\setminus\{ 0,2n\}$, $C^{odd}_{2,3}=\{1,3,5,...,n,2n+1,2n+3,2n+5,...,3n\}$, $C^{\dag}_{odd}=\{1,3,5,...,n-1,2n+1,2n+3,2n+5,...,3n-1 \}$.\\

For each $k$ and $h$, put $\widehat{\chi}_{k}=\chi_{k}\mid_{\widehat{SD}_{8n}}$, $\widehat{\chi'}_{k}=\chi'_{k}\mid_{\widehat{SD}_{8n}}$ and $\widehat{\psi}_{k}=\psi_{k}\mid_{\widehat{SD}_{8n}}$, $\widehat{\psi'}_{k}=\psi'_{k}\mid_{\widehat{SD}_{8n}}$.  Moreover, for odd prime $p$ and $4n=lp^{t}$ such that $l$ is not divisible by $p$, we define $E:=\{2,4,6,...,\frac{l}{2}-2 \}$, $O^{\epsilon}_{1}:=\{1,3,5,...,\frac{l}{4}-\epsilon \}$, $O^{\epsilon}_{2}:=\{\frac{l}{2}+1, \frac{l}{2}+3,...,\frac{l}{2}+\frac{l}{4}-\epsilon \}$, where $\epsilon=1$ if $n$ is even and $\epsilon=2$ if $n$ is odd.
\begin{prop}\label{irreducible and distinct Brauer SD} Let IBr$(SD_{8n})$ be the set of all distinct irreducible Brauer characters of $SD_{8n}$.  Then, \\ $IBr(SD_{8n})=\left\{
                                            \begin{array}{ll}
                                              \{\widehat{\chi}_{k}, \widehat{\psi}_{jp^{t}} \mid 0\leq k \leq3, j \in E\cup O_{1}^{1}\cup O^{1}_{2} \}, & \hbox{if $p\neq2$ and $n$ is even;} \\
                                              \{\widehat{\chi'}_{k}, \widehat{\psi'}_{jp^{t}} \mid 0\leq k \leq7, j \in E\cup O_{1}^{2}\cup O^{2}_{2} \}, & \hbox{if $p\neq2$ and $n$ is odd;} \\
                                              \{\widehat{\chi}_{0}, \widehat{\psi}_{jp^{t}} \mid 0\leq j \leq \frac{l-1}{2} \}, & \hbox{if $p=2$ and $n$ is even.}\\
\{\widehat{\chi'}_{0}, \widehat{\psi'}_{jp^{t}} \mid 0\leq j \leq \frac{l-1}{2}  \}, & \hbox{if $p=2$ and $n$ is odd.}
                                            \end{array}
                                          \right.
$

\end{prop}
\begin{proof} We first note that the restriction of a character of $SD_{8n}$ to $\widehat{SD}_{8n}$ is a Brauer character and the order of the set IBr$(SD_{8n})$ is the number of $p$-regular classes of $SD_{8n}$.   Also, since $SD_{8n}$ is solvable, by Fong-Swan theorem, any element in IBr$(SD_{8n})$ is the restriction of an ordinary irreducible character of $SD_{8n}$.

Each $\widehat{\chi}_{k}$'s and $\widehat{\chi'}_{k}$'s are obviously irreducible and clearly distinct, by the character tables above.  For characters of dimension two, $\widehat{\psi}_{jp^{t}}$ where $p$ is an odd prime and $n$ is even, we claim that those are irreducible.  We suppose for a contradiction that $\widehat{\psi}_{jp^{t}}=\widehat{\chi}_{i}+\widehat{\chi}_{k}$ for some $j \in E \cup O^{1}_{1}\cup O^{2}_{2}$ and $0\leq i,k \leq 3$.  Evaluating both sides at $a^{2p^{t}}$ yields that $$2\cos \frac{jp^{t}\cdot 2p^{t}\pi}{2n}=2.$$  That is $\cos \frac{4jp^{t}\pi}{l}=1$,  so $2j$ is a multiple of $l$.   However, since $2j<l $ for $j \in E\cup O^{1}_{1}$ and $l<2j<2l$ for $j \in O^{1}_{2}$, this is a contradiction.  We use similar arguments to show that all the remaining cases,  $\widehat{\psi}_{jp^{t}}$'s are irreducible.

Next, we aim to show that all elements in IBr$(SD_{8n})$ shown in the proposition are distinct.   For the case odd prime $p$ and even $n$, we suppose that $\widehat{\psi}_{jp^{t}}=\widehat{\psi}_{ip^{t}}$ for some $i,j \in E\cup O_{1}^{1}\cup O^{1}_{2}$.  It is clear (by the character table) that $i,j$ either both are even or both are odd.  If $i,j$ are even, we evaluate both sides at $a^{p^{t}}$ and then we get $$\sin\frac{p^{t}(i+j)\pi}{l}\sin\frac{p^{t}(j-i)\pi}{l}=0.$$  Since $gcd(l,p^{t})=1$ and $\frac{i+j}{l}$ and $\frac{j-i}{l}$ can not be positive integers for each $i,j \in E$, $i=j$.   If $i,j$ are odd, we evaluate both sides at $a^{p^{t}}$, and then we get $$\sin\frac{p^{t}(j-i)\pi}{l}\cos\frac{p^{t}(j+i)\pi}{l}=0.$$  Since $gcd(l,p^{t})=1$ and $\frac{i+j}{l}\neq \frac{l}{2},\frac{3l}{2}$  for $i,j \in O^{1}_{1}\cup O^{1}_{2}$ and $\frac{j-i}{l}$ can not be positive integer, $i=j$.  Again, similar arguments work for all the remaining cases.
\end{proof}

\section{existence of an o-basis  for the class of tensors using a Brauer
character of the $SD_{8n}$}
In the following theorem, we denote
$$
 \epsilon = \left\{\begin{array}{ll} 3~~ \hspace{0.8cm} \text{if} \hspace{0.8cm} p\neq 2,~~ n~~ \text{even} \\
7 ~~ \hspace{0.8cm} \text{if} \hspace{0.8cm} p\neq 2,~~n~~ \text{odd} \\
1 ~~\hspace{0.8cm} \text{if} \hspace{0.8cm}  p= 2 .\\
\end{array}\right.
$$

\begin{thm}
 \label{main1}
 Let $\dim V > 2$, $G=SD_{8n} $, $0 \leq j\leq \epsilon$,  and put $\phi=\hat{\chi}_j$ or $\hat{\chi'}_j $ . Then, space $V_\phi(G)$ has an o-basis if and only if $p=2$ or $4n$ is not divisible by $p$.
\end{thm}
\begin{proof}
 If $p=2$ then $\widehat{G}=\langle a^{p^t}\rangle$. Since $\widehat{G}$ is a subgroup of $G$ and $\phi$ is constant on $\widehat{G}$, by Proposition \ref{prop111},  $V_\phi(G)$ has an o-basis.
Assume $p\neq 2$ and $4n$ is not divisible by $p$. Then $\hat{G}=G$ and consequently, these characters will be ordinary linear characters. Thus $V_\phi(G)$ has an o-basis.\\

Conversely, suppose that $p\neq 2$ and $4n$ is divisible by $p$. We aim to show that $V_\phi(G)$ does not have an o-basis by showing that there exists $\alpha \in \Gamma^{4n}_k$ such that $V^{\phi}_{\alpha}(G)$ does not have an o-basis and then apply Theorem \ref{sum} to conclude the results.\\

Let $\alpha=(1,2,2,...,2,3)$.  Since $\dim V>2$ and $4n\geq 4$, $\alpha \in \Gamma^{4n}_{\dim V}$.  We also choose a representative $\Delta$ so that $\alpha \in \Delta$.  We observe that to fix $\alpha$, each $\sigma \in G$ must fix the first and the last position of $\alpha$.  It is clear that element of the form $a^{k}$ satisfying the condition is only $e$.  For elements of the form $ba^{k}$, they must satisfy $T(ba^{k})(1)=1$ and $T(ba^{k})(4n)=4n$.  By using $T(ba^{k})(t)=\overline{(2n-1)(k+t)}$, we conclude that $G_{\alpha}=\{ e\}$.  Since $\phi$ is a restriction of a linear character and $G_{\alpha}=\{ e \}$, by Corollary \ref{c2}, to show that $\langle e^{\phi}_{\alpha \sigma }, e^{\phi}_{\alpha } \rangle \neq 0$ for each $\sigma \in G$, it suffices to show that $A=\{\mu \in \widehat{G}| e \in \sigma\mu^{-1}\widehat{G}\}\neq\emptyset$.  This is simple because $p\neq 2$ and $4n$ is divisible by $p$, so $\widehat{G}=\{ba^{k}| 0\leq k < m \}=\widehat{G}^{-1}$ and then $A=\{\mu \in \widehat{G}| \sigma \in \widehat{G}\mu\}$.  Since $e \in \sigma\mu^{-1}\widehat{G}$ if and only if $\sigma \in \widehat{G}^{-1}\mu=\widehat{G}\mu$ and thus for arbitrary $0\leq k < 4n$, we have $a^{k}=ba^{0}ba^{k}\in \widehat{G}^{2}$ and $ab^{k}=a^{0}ba^{k}\in \widehat{G}^{2}$, so $G\subseteq \widehat{G}^{2}$.  That is $A\neq\emptyset$.  So,
\begin{equation}\label{t2111}
    \langle e^{\phi}_{\alpha \sigma }, e^{\phi}_{\alpha } \rangle \neq 0 \hbox{ for each } \sigma \in G.
\end{equation}

Next, to show that $\{ e^{\phi}_{\alpha a},  e^{\phi}_{\alpha}\}\subseteq V^{\phi}_{\alpha}(G)$ is a linearly independent set, we set $e^{\phi}_{\alpha}= \sum_{\delta}c_{\delta}e_{\delta}$ and $e^{\phi}_{\alpha a}= \sum_{\delta}d_{\delta}e_{\delta}$, as $\{e_{\delta}| \delta \in \Gamma^{m}_k \}$ forms a basis for $V^{\otimes m}$.  Since $\widehat{G}^{-1}=\widehat{G}$, $$e^{\phi}_{\alpha}=\frac{\phi(1)}{|\widehat{G}|}\sum_{\sigma\in \widehat{G}}\phi(\sigma^{-1})e_{\alpha\sigma}.$$
Since $G_{\alpha}=\{e\}$, the elements $\alpha \sigma$ with $\sigma \in G$ are distinct.  Also, since $a \notin \widehat{G}$, $\alpha\sigma \neq\alpha a$ for all $\sigma \in \widehat{G}$, which yields that $c_{ a}=0$. On the other hand, $G_{\alpha a}=a^{-1}G_{\alpha}a=\{e\}$, so for $\a \in \widehat{G}$, $(\alpha a)\sigma=\alpha a$ if and only if $\sigma=e$.  This implies that $d_{ a}=\frac{1}{|\widehat{G}|}\neq c_{\alpha a}=0$.  Thus $\{ e^{\phi}_{\alpha a},  e^{\phi}_{\alpha}\}\subseteq V^{\phi}_{\alpha}(G)$ is a linearly independent set and hence $\dim V^{\phi}_{\alpha}(G)\geq 2$.

By Proposition \ref{rebasis}, if $V^{\phi}_{\alpha}(G)$ were to have an o-basis, then it would have an o-basis containing $e^{\phi}_{\alpha}$, but, by (\ref{t2111}), this is not the case.  So, $V^{\phi}_{\alpha}(G)$ does not have an o-basis, which completes the proof.

\end{proof}

\begin{rem}
 \label{remark1}
  Theorem \ref{main1} shows that if $dim V >2$,  unlike the case for an irreducible character, it is possible that $V_\phi(G)$ has no o-basis when $\phi$ is a linear Brauer character. This holds when $dim V =2$ as well. To observe this, we let $p\neq 2$ and $4n$ is divisible by $p$ and $\phi=\hat{\chi}_{0}$ or $\hat{\chi'}_{0} $ . Consider $\alpha=(1,2,..,2,2)  \in \Gamma^{4n}_{\dim V}$.  Thus, for such $\alpha$, we have $G_{\alpha}=\{1,a^{2n+2}b\}$. Now by similar calculations done in  Theorem \ref{main1} we have  $\langle e^{\phi}_{\alpha \sigma_1 }, e^{\phi}_{\alpha \sigma_2} \rangle \neq 0$.
\end{rem}
 For the remaining of this section,  we denote
$$ \Pi=\left\{
         \begin{array}{ll}
           \{jp^t  | j \in E\cup O^1_1\cup O^1_2\}, & \hbox{if $p\neq2$ and $n$ even;} \\
           \{jp^t  | j \in E\cup O^2_1\cup O^2_2\}, & \hbox{if $p\neq2$ and $n$ odd;} \\
           \{ jp^t  | 0\leq j \leq \frac{l-1}{2} \}, & \hbox{if $p=2$.}
         \end{array}
       \right.
 $$
For $V_{\phi}(SD_{8n})$, where $\phi=\hat{\psi}_h, \hat{\psi'}_h$ such that $h \in \Pi$ is even,  the condition for the existence can be obtained in the same manner as $T_{4n}$ and $D_m$. This is because $\phi$ has a cyclic support and all values in the character tables of those groups  are consistence on $C=<a^r>$ if $h$ is even.  Then, we have;
\begin{thm}Let $G=SD_{8n}$, where $4n=lp^t$ with $l$ an integer not divisible by $p$ and let $\phi=\hat{\psi}_h, \hat{\psi'}_h$ such that $h \in \Pi$ be even.  Then,   $V_{\phi}(G)$ has an o-basis if and only if $\nu_2(\frac{2h}{l})<0$. Also, for each $\sigma \in C$,  $ e^{\phi}_{\gamma\sigma}\ \neq 0 \hbox{  if and only if  }\frac{ht_\gamma}{l}\in\mathbb{Z}.$
\end{thm}

For the case  $h \in \Pi$ is odd, we first compute the dimension of $V_{\gamma}^{\phi}(SD_{8n})$.
\begin{prop}\label{dimen SD}Let $G=SD_{8n}$, where $4n=lp^t$ with $p\nmid l$ and let $\phi=\hat{\psi}_h, \hat{\psi'}_h$ such that $h \in \Pi$ be odd. For $\gamma\in \Delta$ such $C\cap G_\gamma=<a^{t_\gamma p^t}>$, where $t_\gamma=\frac{l}{|C\cap G_\gamma|}$, then
$$\dim(V_\gamma^\phi(C))=\left\{
                           \begin{array}{ll}
                             4, & \hbox{if $\frac{ht_\gamma }{l}\in \mathbb{Z}$;} \\
                             0, & \hbox{if $\frac{ht_\gamma }{l}\notin \mathbb{Z}$.}
                           \end{array}
                         \right.
  $$
\end{prop}
\begin{proof}  Since $G$ has cyclic support with $C=<a^{p^t}>$ and $C\cap G_\gamma=<a^{t_\gamma p^t}>$, for each $\phi$, we can compute the dimension by using Proposition \ref{dimprop}.  By character tables, we compute $v_j=\sum_{h\in C\cap G_\gamma}\phi(h\tau^{t_\gamma-j})$, for the different case of $j$ and $t_\gamma$.
 If $t_\gamma$ is odd, then, for even $j$,
$$\begin{array}{ccl}
    v_j& = & \sum_{m=1}^{l/t_{\gamma}}\phi(r^{(mt_\gamma-j)p^t}) \\
     & = &i \sum_{k=1}^{l/2t_{\gamma}}2\sin\left((2k-1)(\frac{2ht_\gamma}{l})\pi-(\frac{2hj}{l})\pi  \right) \\ &+&\sum_{k=1}^{l/2t_{\gamma}}2\cos\left((2k)(\frac{2ht_\gamma}{l})\pi-(\frac{2hj}{l})\pi  \right)\\
     &=&0,
  \end{array}
  $$
and for odd $j$,
$$\begin{array}{ccl}
    v_j& = & \sum_{m=1}^{l/t_{\gamma}}\phi(r^{(mt_\gamma-j)p^t}) \\
     & = & \sum_{k=1}^{l/2t_{\gamma}}2\cos\left((2k-1)(\frac{2ht_\gamma}{l})\pi-(\frac{2hj}{l})\pi  \right) \\ &+&i\sum_{k=1}^{l/2t_{\gamma}}2\sin\left((2k)(\frac{2ht_\gamma}{l})\pi-(\frac{2hj}{l})\pi  \right)\\
     &=&0,
  \end{array}
  $$
since $\frac{2ht_\gamma}{l} \notin \mathbb{Z}$ (because $h, t_\gamma$ are odd and then $4\mid l$).  So, if $t_\gamma$ is odd (i.e. $\frac{ht_\gamma }{l}\notin \mathbb{Z}$), then $d_\gamma=t_\gamma$ and thus $\dim(V_\gamma^\phi(C))=t_\gamma-t_\gamma=0$.

Similarly, if $t_\gamma$ is even, we compute that
$$v_j=\left\{
        \begin{array}{ll}
          0, & \hbox{$\frac{ht_\gamma }{l}\notin \mathbb{Z}$;} \\
          -\frac{2il}{t_\gamma}\sin(\frac{2hj}{l}\pi), & \hbox{$\frac{ht_\gamma }{l}\in \mathbb{Z}$ and $j$ odd} \\
          \frac{2l}{t_\gamma}\cos(\frac{2hj}{l}\pi), & \hbox{$\frac{ht_\gamma }{l}\in \mathbb{Z}$ and $j$ even.}
        \end{array}
      \right.
  $$
So, if $t_\gamma$ is even and $\frac{ht_\gamma }{l}\notin \mathbb{Z}$, then $d_\gamma=t_\gamma$ and thus $\dim(V_\gamma^\phi(C))=t_\gamma-t_\gamma=0$.
For $d_\gamma$ in which $t_\gamma$ is even and $\frac{ht_\gamma}{l}\in\mathbb{Z}$, we have that $\sum^{t_\gamma-1}_{j=0}v_je^{\frac{2\pi sji}{t_\gamma}}=0$ if and only if
$$\frac{l}{t_\gamma}\left[ \sum^{\frac{t_\gamma}{2}-1}_{k=0}2\cos((\frac{4hk}{l})\pi)\cos((\frac{4sk}{t_\gamma})\pi) + \sum^{\frac{t_\gamma}{2}-1}_{k=0}2\sin((\frac{2h(2k+1)}{l})\pi)\sin((\frac{2s(2k+1)}{t_\gamma})\pi) \right],$$
and
$$\frac{il}{t_\gamma}\left[ \sum^{\frac{t_\gamma}{2}-1}_{k=0}2\cos((\frac{4hk}{l})\pi)\sin((\frac{4sk}{t_\gamma})\pi) - \sum^{\frac{t_\gamma}{2}-1}_{k=0}2\sin((\frac{2h(2k+1)}{l})\pi)\cos((\frac{2s(2k+1)}{t_\gamma})\pi) \right],$$
are simultaneously zero.  Since $\frac{ht_\gamma}{l}\in\mathbb{Z}$, the second sum is always zero and the first sum is zero for each $0\leq s< t_\gamma$ except for $2\left(\frac{h}{l}+\frac{s}{t_\gamma}\right) \in \mathbb{Z}$ or $2\left(\frac{h}{l}-\frac{s}{t_\gamma}\right) \in \mathbb{Z}$.  Since $0\leq \frac{s}{t_\gamma}< 1$, $$\hbox{$\frac{2h}{l}\leq 2\left(\frac{h}{l}+\frac{s}{t_\gamma}\right)<\frac{2h}{l}+2$ and  $\frac{2h}{l}-2<2\left(\frac{h}{l}-\frac{s}{t_\gamma}\right)\leq\frac{2h}{l}$}. $$  Precisely, if $s$ belongs to
$$
    \{ s_1:=\frac{t_\gamma}{2}\lceil\frac{2h}{l}\rceil-\frac{ht_\gamma}{l},  s_2:=\frac{t_\gamma}{2}\lceil\frac{2h}{l}\rceil-\frac{ht_\gamma}{l}+\frac{t_\gamma}{2},
     s_3:=\frac{ht_\gamma}{l}-\frac{t_\gamma}{2}\lfloor\frac{2h}{l}\rfloor, s_4:=\frac{ht_\gamma}{l}-\frac{ t_\gamma}{2}\lfloor\frac{2h}{l}\rfloor+\frac{t_\gamma}{2}\}
   $$
then the first sum will be not zero.  Here, $\lceil r\rceil$ and $\lfloor r \rfloor$ are the ceil function and floor function of the real number $r$ respectively.  Since $t_\gamma >0$ and $\frac{4h}{l} \notin \mathbb{Z}$ for each odd $h \in \Pi$, $s_1, s_2, s_3, s_4$ are all distinct.  Hence $d_\gamma= t_\gamma -4$ and thus the results follow.
\end{proof}
Now, we have;
\begin{thm}  Let $G=SD_{8n}$, where $4n=lp^t$ with $p\nmid l$ and let $\phi=\hat{\psi}_h, \hat{\psi'}_h $, where $ h \in \Pi$ be odd.  If $\dim(V) > 1$, then,   $V_{\phi}(G)$  does not have an o-basis.  Also, for each $\sigma \in C$,  $ e^{\phi}_{\gamma\sigma}\ \neq 0 \hbox{  if and only if  }\frac{ht_\gamma}{l}\in\mathbb{Z}.$
\end{thm}
\begin{proof}  By Proposition \ref{separated lemma}, it is enough to focus on $V^{\phi}_{\gamma}(C)$.  Also, in the proof of Proposition \ref{prop dim}, we have $V_\gamma^\phi(C)  =  \langle e^{\phi}_{\gamma \sigma_j} | j=1,2,...,t_\gamma \rangle$, where $\sigma_j=a^{jp^t}$ and $t_{\gamma}=\frac{\mid C\mid}{\mid C\cap G_\gamma  \mid}$.  By (\ref{inner C}) and the character tables, we compute that, for even $i, j$ such $1\leq i, j \leq t_\gamma$,
$$ \begin{array}{ccl}
      \langle e^{\phi}_{\gamma \sigma_i}, e^{\phi}_{\gamma \sigma_j} \rangle =0& \Longleftrightarrow & \sum_{g\in C}\phi(g\sigma_i)\overline{\phi(g\sigma_j)}=0 \\
      & \Longleftrightarrow & \sum^{l-1}_{k=0}\phi(a^{(k+i)p^{t}})\overline{\phi(a^{(k+j)p^{t}})}=0\\
&\Longleftrightarrow&2\left[\sum^{\frac{l}{2}-1}_{k=0}2\cos((2k+i)\frac{2h}{l}\pi)\cos((2k+j)\frac{2h}{l}\pi)\right]\\
&+&2\left[\sum^{\frac{l}{2}-1}_{k=0}2\sin((2k+1+i)\frac{2h}{l}\pi)\cos((2k+1+j)\frac{2h}{l}\pi)\right]=0\\
&\Longleftrightarrow&2\left[\sum^{\frac{l}{2}-1}_{k=0}\cos((4k+i+j)\frac{2h}{l}\pi)+\sum^{\frac{l}{2}-1}_{k=0}\cos((i-j)\frac{2h}{l}\pi)\right]\\
&+&2\left[\sum^{\frac{l}{2}-1}_{k=0}\cos((i-j)\frac{2h}{l}\pi)-\sum^{\frac{l}{2}-1}_{k=0}\cos((4k+i+j)\frac{2h}{l}\pi)\right]=0\\
&\Longleftrightarrow&2l\cos((i-j)\frac{2h}{l}\pi)=0, \hbox{  (since $\frac{4h}{l} \notin \mathbb{Z}$).}
   \end{array}
 $$
Similar arguments work well for the remaining cases.  Thus, we can conclude that
\begin{equation}\label{hodd}
     \langle e^{\phi}_{\gamma \sigma_i}, e^{\phi}_{\gamma \sigma_j} \rangle =0 \Longleftrightarrow \left\{
                                                                                                   \begin{array}{ll}
                                                                                                     \cos((i-j)\frac{2h}{l}\pi)=0, & \hbox{if $i, j$ are both even or both odd;} \\
                                                                                                     \sin((i-j)\frac{2h}{l}\pi)=0, & \hbox{if ortherwise.}
                                                                                                   \end{array}
                                                                                                 \right.
\end{equation}

We consider $\gamma=(1,2,2,...,2) \in \Gamma^{4n}_{\dim(V)}$.   Since $\dim(V)>1$, $\gamma \in \Delta$ and it is not hard to see that $G_\gamma=\{e\}$.  So, $t_\gamma= l$ and then $\frac{ht_\gamma}{l} \in \mathbb{Z}$.  By Proposition \ref{dimen SD}, $\dim(V^{\phi}_{\gamma}(C))=4$.
Thus, if $V^{\phi}_{\gamma}(C)$ has an o-basis, then there exist distinct $1\leq i_1, i_2, i_3, i_4 \leq t_\gamma$ such that $\{e^{\phi}_{\gamma \sigma_{i_1}}, e^{\phi}_{\gamma \sigma_{i_2}}, e^{\phi}_{\gamma \sigma_{i_3}}, e^{\phi}_{\gamma \sigma_{i_4}},\}$ forms an o-basis.   Since $h$ is odd and $4\mid l$, $\nu(\frac{2h}{l})= -k$, for some positive integer $k$.  Hence, if there are at least three of $ i_1, i_2, i_3, i_4$ which are all even or all odd, say $ i_1, i_2, i_3$, then, by (\ref{hodd}), there must exist odd integers $o_1,o_2, o_3$ such that
$$ \begin{array}{ll}
     i_1-i_2 & =o_1\cdot 2^{k-1} \\
     i_1-i_3 & =o_2\cdot 2^{k-1} \\
     i_2-i_3 & =o_3\cdot 2^{k-1}.
   \end{array}
  $$
 This implies that $o_2-o_3=o_1$, which is a contradiction.  If there are exactly two of $ i_1, i_2, i_3, i_4$ which are all odd, say $ i_1, i_2$, then, by (\ref{hodd}), there must exist integer $s$ such that $i_1-i_3  =s\cdot 2^{k} $.   This implies that $i_1=i_3+s\cdot 2^{k}$ is even (because $i_3$ is even), which is a contradiction.  Therefore, $V^{\phi}_{\gamma}(C)$ does not have an o-basis and by Proposition \ref{separated lemma}, we finish the proof for the first statement.  The second statements is a consequence of Proposition \ref{dimen SD} and a direct calculation as in Proposition \ref{nonzeroT}.
\end{proof}

\section{Acknowledgements}

The authors are so grateful to Professor Hjalmar Rosengren for valuable comments and for reviewing earlier drafts very carefully.  The second author would like to thank Naresuan University for financial supports on the project R2558C240.

\bigskip

\address Mahdi Hormozi \\
{ Department of Mathematical Sciences, Division of Mathematics,\\ Chalmers University
of Technology and University of Gothenburg,\\ Gothenburg 41296, Sweden}\\
\email{hormozi@chalmers.se}\\

\address Kijti Rodtes \\
{ Department of Mathematics, Faculty of Science,\\ Naresuan University \\ Phitsanulok 65000, Thailand}\\
\email{kijtir@nu.ac.th}\\
\


\begin{thebibliography}{20}
\bibitem{BPR}
C. Bessenrodt, M.R. Pournaki, and A. Reifegerste, \emph{A note on the orthogonal basis of a certain full symmetry class of tensors}, Linear Algebra Appl. 370 (2003), pp. 369--374
\bibitem{CS1}
H.F. da Cruz and J.A. Dias da Silva, \emph{Equality of immanantal decomposable tensors}, Linear Algebra Appl. 401 (2005), pp. 29--46.
\bibitem{CS2}
H.F. da Cruz and J.A. Dias da Silva, \emph{Equality of immanantal decomposable tensors, II}, Linear Algebra Appl. 395 (2005), pp. 95--119.
\bibitem{DP}
M. R. Darafsheh and  M. R. Pournaki, \emph{On the orthogonal basis of the symmetry classes of tensors associated with the dicyclic group}, Linear and Multilinear Algebra 47 (2000), no. 2, pp. 137--149
\bibitem{DP2}
M. R. Darafsheh, N. S. Poursalavati, \emph{On the existence of the orthogonal basis of the symmetry classes of tensors associated with certain groups}, SUT J. Math., Vol. 37, no. 1 (2001), pp. 1--17
\bibitem{D}
J.A. Dias da Silva,  \emph{Colorings and equality of tensors}, Linear Algebra Appl. 342 (2002), pp. 79--91.
\bibitem{ST}
J.A. Dias da Silva and Maria M. Torres,  \emph{On the orthogonal dimension of orbital sets}, Linear Algebra Appl. 401 (2005), pp. 77--107.
\bibitem{F1}
A. Fonseca,  \emph{On the equality of families of decomposable symmetrized tensors}, Linear Algebra Appl. 293 (1999), pp. 1--14.
\bibitem{F2}
R. Freese, \emph{Inequalities for generalized matrix functions based on arbitrary characters}, Linear Algebra Appl. 7 (1973), pp. 337--345.
\bibitem{HK}
R. R. Holmes and A. Kodithuwakku, \emph{Orthogonal bases of Brauer symmetry classes of tensors for the dihedral group}, Linear and Multilinear Algebra,  61 (2013), no. 8, pp. 1136--1147.
\bibitem{HORO}
M. Hormozi and K. Rodtes, \emph{ Symmetry classes of tensors associated with the Semi-Dihedral groups $SD_{8n}$},
Colloquium Mathematicum, Vol. 131 (2013), no.1, 59--67.
\bibitem{Kra and Simanca}
I. Kra and S.R. Simanca, \emph{On circulant matrices}, Notices of the AMS, Vol. 59 (2012), no. 3, 368-377.
\bibitem{Bra1}
A. Laradji,   \emph{On relative height zero Brauer characters}, Osaka J. Math., Vol. 50  (2013),  no. 3, 591--605.
\bibitem{Me}
R. Merris,\emph{ Multilinear Algebra}, Gordan and Breach Science Publishers, Amsterdam, 1997.
 \bibitem{Bra2}
  G. Navarro, B. Sp\"{a}th, P. H. Tiep, \emph{On fully ramified Brauer characters}, Adv. Math.,  Vol. 257  (2014), 248--265.
\bibitem{Bra3}
 G. Navarro, P. H. Tiep,  \emph{Brauer characters and rationality}, Math. Z.,  Vol. 276  (2014),  no. 3-4, 1101--1112.
\bibitem{Bra4}
B. Sambale, \emph{Blocks with defect group $Q_{2^n}\times C_{2^m}$ and $SD_{2^n}\times C_{2^m}$}, Algebr. Represent. Theory,  Vol. 16  (2013),  no. 6, 1717--1732.
\bibitem{Bra5}
 B. Sambale,  \emph{Blocks with central product defect group $D_{2^n}\ast C_{2^m}$}, Proc. Amer. Math. Soc., Vol.  141  (2013),  no. 12, 4057--4069.
 \bibitem{Bra6}
  F. Schaeffer and A.  Amanda, \emph{Cross-characteristic representations of $Sp_6(2^a)$ and their restrictions to proper subgroups}, J. Pure Appl. Algebra, Vol. 217  (2013),  no. 8, 1563--1582.
\bibitem{Bra7}
H. P. Tong-Viet, \emph{Finite groups whose irreducible Brauer characters have prime power degrees},
Israel J. Math.,   Vol. 202  (2014),  no. 1, 295--319.
\bibitem{Bra8}
 J. Wang, J. Fan and N. Du, \emph{Super-$\pi$-Brauer characters and super-$\pi$-regular classes}, Comm. Algebra,  Vol. 42  (2014),  no. 9, 4102--4109.
\bibitem{Bra9}
A. Watanabe, \emph{The number of irreducible Brauer characters in a $p$-block of a finite group with cyclic hyperfocal subgroup},
 J. Algebra,  Vol. 416  (2014), 167--183.

\end{thebibliography}
\end{document}